\documentclass[12pt,a4paper,reqno]{amsart} %Fr o m detta
\usepackage[latin1]{inputenc}
\usepackage[english]{babel}
\usepackage{amssymb}
\usepackage{latexsym}
\usepackage{amsmath}
\usepackage{amsthm}
\usepackage{mathrsfs}
\usepackage{calc}

\newcommand{\RR}{\mathbf R^n}
\newcommand{\rr}[1]{\mathbf R^{#1}}

\numberwithin{equation}{section}

\newtheorem{thm}{Theorem}
\numberwithin{thm}{section}
\newtheorem{lemma}[thm]{Lemma}
\newtheorem{ex}[thm]{Example}
\newtheorem{anm}[thm]{Remark}
\author{Karoline Johansson}

\title{A counter example on nontangential convergence for oscillatory
integrals} 

\begin{document}

\begin{abstract}

Consider the solution of the time-dependent Schr{\"o}\-dinger
equation with initial data $f$. It is shown in \cite{artikel} that
there exists $f$ in the Sobolev space $H^s(\RR), \; s=n/2$ such
that tangential convergence can not be widened to convergence
regions. In this paper we show that the corresponding result holds
when $-\Delta_x$ is replaced by an operator $\varphi(D)$, with
special conditions on $\varphi$. 
\end{abstract}
\maketitle

%===================================================================================================================

\section{Introduction}

In this paper we generalize previous work by Sj{\"o}gren and
Sj{\"o}lin \cite{artikel} about non-existence of non-tangential
convergence for the solution $u=S^{\varphi}f$ to the generalized time-dependent
Schr{\"o}dinger equation
\begin{equation}\label{generaliserad}
(\varphi(D)+i\partial_t )u =0,
\end{equation}
with the initial condition $u(x,0)=f(x)$. Here $\varphi$ should be
real-valued and its radial derivatives of first and second order ($\varphi' =\varphi'_r$ and $\varphi'' =\varphi''_{rr}$) should be continuous, outside a compact set containing origin. Furthermore, we will require some appropriate conditions on the growth $\varphi'$ and $\varphi''$. (See \eqref{funktionerna} and \eqref{deriverade
funktioner} for exact conditions on $\varphi$.) In particular the
function $\varphi(\xi)=|\xi|^a$ will satisfy these conditions, for
$a> 1$.

\par
\vspace{0.5 cm}

 For $\varphi(\xi)=|\xi|^2$ it was shown in
\cite{artikel} that there exists a function $f$ such
that near the vertical line $t \mapsto (x,t)$ through an arbitrary
point $(x,0)$ there are points accumulating at $(x,0)$ such
that the solution of equation \eqref{generaliserad} takes values
far from $f$. This means that the solution of the time-dependent Schr{\"o}dinger equation with
initial condition $u(x,0)=f(x)$ does not converge non-tangentially
to $f$. Therefore we can not consider regions of convergence.

\par
\vspace{0.5 cm}

In this paper, we prove that this property holds for more general
functions $\varphi(\xi)$ of the type described above. In the proof we use some ideas by Sj{\"o}gren and Sj{\"o}lin in \cite{artikel} in combination with new estimates, to construct a counter example. Some ideas can also be found in Sj{\"o}lin
\cite{{LP}, {Counter}} and Walther \cite{{Sharp},{Sharpmax}}, and some related results are given in Bourgain \cite{Bourgain}, Kenig, Ponce and Vega
\cite{Kenig}, and Sj{\"o}lin \cite{{ref2},{Hom}}.

\par
\vspace{0.5 cm}

Existence of regions of convergence has been studied before for
other equations. For example, Stein and Weiss consider in
\cite[Theorem $3.16$]{SteinWeiss} Poisson integrals acting on
Lebesgue spaces. These
operators are related to the operator $S^{\varphi}$.
\par
\vspace{0.5 cm}

%============================================================================================
For an appropriate function $\varphi$ on $\RR$, let $S^{\varphi}$
be the operator acting on functions $f$ defined by
\begin{equation}\label{allmSaf} f \mapsto
\mathcal{F}^{-1}(\exp(i t \varphi(\xi))\mathcal{F}f),
\end{equation}
where $\mathcal{F}f$  is the Fourier transform of $f$, which takes the form
\begin{equation}
\widehat{f}(\xi)=\mathcal{F}f(\xi)\equiv\int_{\RR}e^{-i x \cdot \xi}f(x)\, d x,
\end{equation}
when $f\in L^1(\RR)$. This means that, if $\widehat{f}$ is an
integrable function, then $S^{\varphi}$ in \eqref{allmSaf} takes
the form
\begin{equation}\label{Saf}
S^{\varphi} f(x,t) = \frac{1}{(2 \pi)^n} \int_{\RR} {e^{i x\cdot
\xi}e^{it \varphi(\xi)}\widehat{f}(\xi)}\, d \xi,\quad x \in
\RR,\quad t \in \mathbf{R}.
\end{equation}

\par
\vspace{0.5 cm}

If $\varphi(\xi)=|\xi|^2$ and $f$ belongs to the Schwartz class
$\mathcal{S}(\RR)$, then $S^{\varphi} f$ is the solution to the
time-dependent Schr{\"o}dinger equation $(-\Delta_x +i\partial_t)u
=0$ with the initial condition $u(x,0) = f(x)$.

\par
\vspace{0.5 cm}
For more general
appropriate $\varphi$, for which the equation
\eqref{generaliserad} is well-defined, the expression $S^{\varphi}
f$ is the solution to the generalized time-dependent
Schr{\"o}dinger equation \eqref{generaliserad} with the initial
condition $u(x,0) = f(x)$. Note here that $S^{\varphi}f$ is
well-defined for any real-valued measurable $\varphi$ and $f\in \mathcal{S}$. On the other hand, it might be difficult to interpret \eqref{generaliserad} if for example $\varphi\not\in L^1_{loc}$.
\par
\vspace{0.5 cm}

In order to state the main result we need to specify the
conditions on $\varphi$ and give some definitions. The function $\varphi$ should satisfy the
conditions
\begin{gather}\label{funktionerna}
\liminf_{r\to \infty}(\inf_{|\omega|=1}|\varphi'(r,\omega)|)=\infty,
\end{gather}
and
\begin{gather}\label{deriverade funktioner}
\sup_{r\geq
R}\Big(\sup_{|\omega|=1}\frac{r|\varphi''(r,\omega)|}{|\varphi'(r,\omega)|^2(\log
r)^{3/4}}\Big)<C.
\end{gather}

Here $\varphi'(r\omega)=\varphi'(r,\omega)$ denotes the derivative
of $\varphi(r,\omega)$ with respect to $r$, and similarly for higher orders of derivatives.

\par
\vspace{0.5 cm}

We let $H^s(\RR)$ be the Sobolev space
of distributions with $s\in \mathbf{R}$ derivatives in $L^2$.
That is $H^s(\RR)$ consists of all $f\in \mathcal{S}'(\RR)$ such that
\begin{equation}\label{sobolev}
\|f\|_{H^s(\RR)}\equiv \big(\int_{\RR}(1 + |\xi |^2)^s
|\widehat{f} (\xi )|^2  \, d \xi \big)^{1/2} < \infty.
\end{equation}
%

%============================================================================================
%
\begin{thm} \label{a>1} Assume that the function
$\gamma:\mathbf{R}_+\rightarrow\mathbf{R}_+$ is strictly
increasing and continuous such that $\gamma(0)=0$. Let $R>0$, and let $\varphi$
be real-valued functions on $\RR$ such that $\varphi '(r,\omega)$ and $\varphi''(r,\omega)$ are continuous and satisfy \eqref{funktionerna} and \eqref{deriverade funktioner} when $r>R$. Then there
exists a function $f \in H^{n/2}(\RR)$ such that $S^{\varphi} f$
is continuous in $\{(x,t); t>0\}$ and
\begin{equation}\label{Sa till infty}
\limsup_{(y,t)\rightarrow (x,0)} |S^{\varphi} f(y,t)|= +\infty
\end{equation}
for all $x \in \RR$, where the limit superior is taken over those $(y,t)$ for which $|y-x|<\gamma (t)$ and $t>0$.
\end{thm}
%%
%%
%============================================================================================
Here we recall that $\varphi'=\varphi'_r$ and $\varphi'=\varphi''_{rr}$ are the first and second orders radial derivatives of $\varphi$.
When $s>n/2$ no counter example of the form in Theorem \ref{a>1} can be provided, since $S^{\varphi}f(y, t)$ converges to $f(x)$ as $(y,t)$ approaches $(x,0)$ non-tangentially when $f \in H^s(\RR)$. In fact, H{\"o}lder's inequality gives
\begin{gather*}
(2 \pi)^{n}|S^{\varphi}f(x,t)|\leq \int_{\RR}|\widehat{f}(\xi)| \, d\xi\leq
\Big(\int_{\RR} (1+|\xi|^2)^{-s}\, d\xi\Big) \|f\|_{H^s(\RR)},
\end{gather*}
which is finite when $f\in H^s(\RR), \, s>n/2$.
Therefore convergence along vertical lines can be extended
to convergence regions when $s>n/2$ and $f$ belongs to $H^s(\RR)$.
\vspace{0.5 cm}
%============================================================================================

For functions $\varphi$ satisfying 
\begin{equation}\label{nedat begr}
 \inf_{r>R}(\inf_{|\omega|=1} |\varphi'(r,\omega)|)=h>0
\end{equation}
and one of the conditions \eqref{deriverade funktioner} or 
\begin{gather}\label{deriverade funktioner1}
\sup_{r\geq R}\Big(\sup_{|\omega|=1}\frac{r^{\beta}|\varphi''(r,\omega)|}{(\log r)^{3/4}}\Big)<C,
\end{gather}
for some $\beta>0$, we can prove a weaker form of Theorem \ref{a>1}.

\begin{thm}\label{a=1} Assume that the function
$\gamma:\mathbf{R}_+\rightarrow\mathbf{R}_+$ is strictly
increasing and continuous such that $\gamma(0)=0$. Let $R>0$, and let $\varphi$
be real-valued functions on $\RR$ such that $\varphi '(r,\omega)$ and $\varphi''(r,\omega)$ are continuous and satisfy \eqref{nedat begr}, and \eqref{deriverade funktioner} or  \eqref{deriverade funktioner1} when $r>R$. Then for fixed $x\in \RR$ there
exists a function $f \in H^{n/2}(\RR)$ such that $S^{\varphi} f$
is continuous in $\{(x,t); t>0\}$ and
\begin{equation}\label{S1 till infty}
\limsup_{(y,t)\rightarrow (x,0)} |S^{\varphi} f(y,t)|= +\infty,
\end{equation} where the limit superior is taken over those $(y,t)$ for which $|y-x|<\gamma (t)$ and $t>0$.
\end{thm}

\section{Examples and remarks}
In this section we give some examples of functions $\varphi$ for
which Theorem \ref{a>1} holds. In the first example we let
$\varphi$ be a positively homogeneous function of order $a>1$.
%==================================================================================================================
\begin{ex}\label{tidigare res} \textnormal{ Let $a>1$ and $\varphi(\xi)= |\xi|^a$, then $S^{\varphi}f(x,t)$
is the solution to the generalized time-dependent Schr{\"o}dinger
equation $((-\Delta_x)^{a/2} +i\partial_t)u =0$.  By change of
variables to polar coordinates and derivate with respect to $r$ we
see that $\varphi(r,\omega) = r^a$, $\varphi'(r,\omega) =
ar^{a-1}$ and $\varphi''(r,\omega) = a(a-1)r^{a-2}$. We can see
that these derivatives satisfy \eqref{funktionerna} and \eqref{deriverade funktioner}. In particular for $a=2$ this is the solution to the
time-dependent Schr{\"o}dinger equation $(-\Delta_x +i\partial_t)u
=0$ and this case is treated in Sj{\"o}gren and Sj{\"o}lin
\cite{artikel}.}
\end{ex}
In the following example we let $\varphi$ be a sum of positively
homogeneous functions where $a>1$ denote the term of highest
order.

\begin{ex}\label{linjkomb}\textnormal{ For $a>1$, let
\begin{gather}
\varphi(\xi) =\sum_{i=1}^d |\xi|^{a_i}\varphi_{a,i}(1,\omega), \quad a_1<\cdots <a_d=a,
\end{gather}
where
\begin{gather*}
\inf_{\omega}|\varphi_{a,d}(1,\omega)|=h>0 \qquad \text{ and
 }\qquad \|\varphi_{a,i}(1,\cdot)\|_{L^{\infty} (S^{n-1})}<\infty
\end{gather*}
for
each $i\in \{1,2,\dots , d\}$. Here $S^{n-1}$ is the
$n-1$-dimensional unit sphere. By rewriting this into polar
coordinates and differentiate with respect to $r$, we see that in the
first derivative the term $\varphi_{a,i}(1,\omega) r^{a-1}$
dominates the sum and that the second derivative can be estimated
by $C r^{a-2}$, for some constant $C$. These derivatives satisfy
 \eqref{funktionerna} and \eqref{deriverade funktioner}.}
\end{ex}

In the examples at the above we have used functions $\varphi$ such that the modulus of the radial derivative is bounded from below by a positive homogeneous function of order $a-1$ for some $a>1$. 
This condition is not necessary. The hypothesis in the theorem permit a broader class
of functions $\varphi$.
The following example shows that there are functions, which do not grow as fast as a positive homogeneous function of order $a-1$ for any $a>1$, but satisfy the conditions \eqref{funktionerna} and \eqref{deriverade funktioner}.

\begin{ex}\textnormal{ Let $\varphi(\xi)=|\xi|\log |\xi|$, then $\varphi'(r,\omega)=\log r+1$ and $\varphi''(r,\omega)=r^{-1}$ and 
\eqref{funktionerna} and \eqref{deriverade funktioner} are
satisfied.}
\end{ex}

We also allow the dominant part
of the derivative to grow faster than any positively homogeneous
function as long as we have some restrictions on the second
derivative. The conditions are given explicitly in
\eqref{funktionerna} and \eqref{deriverade funktioner}. The
following example contains such functions.

\begin{ex}\label{exponential}\textnormal{ Let $\varphi(\xi)=\varphi(r,\omega)=e^{\mu(\omega)r^\beta}$,
where $\beta>0$ and \\ $\inf_{|\omega|=1} \mu(\omega) =c>0$. These
functions grow faster than $r^a$ for all $a$ and the same is true
for the absolute value of the first and second derivative with
respect to $r$. This can be used to show that
\eqref{funktionerna} and \eqref{deriverade funktioner} are
satisfied.}
\end{ex}

\section{Notations for the proofs}\label{not}
In order to prove Theorems \ref{a>1} and \ref{a=1} we introduce some notations.
Let $B_r(x)$ be the open ball in $\RR$ with center at $x$ and radius
$r$. Numbers denoted by $C, \,c$ or $C'$ may be different at each
occurrence. We let
\begin{equation}\label{deltaj}
\delta _k =\delta_{k,n} \equiv \gamma (1/(k+1)) / \sqrt{n}, \qquad
k \in \mathbf{N},
\end{equation}
where $\gamma$ is the same as in Theorem \ref{a>1} and Theorem \ref{a=1}. Since
$\gamma$ is strictly increasing it is clear that $(\delta
_k)_{k\in \mathbf{N}}$ is strictly decreasing. 
We also let $(x_j)_{j=1}^{\infty}\subset \RR$ be chosen such that 
$x_1, x_2,\dots , x_{m_1}$ denotes all points in $B_{1}(0)\cap
\delta_{1} \mathbf{Z}^n$, $x_{m_1+1},\dots , x_{m_2}$ denotes all points in $B_{2}(0)\cap
\delta_{2} \mathbf{Z}^n$ and generally 
\begin{equation}
\{x_{m_k+1}, \dots , x_{m_{k+1}}\}=B_{k+1}(0)\cap
\delta_{k+1} \mathbf{Z}^n,\qquad \text{ for } k\geq 1.
\end{equation}
Furthermore we choose a strictly decreasing sequence $(t_j)_1^{\infty}$ such that
$
1>t_1>t_2>\cdots > 0
$
and
\begin{equation}
\frac{1}{k+2}<t_j< \frac{1}{k+1},\qquad k\in \mathbf{N},
\end{equation}
for $m_k +1\leq j \leq m_{k+1}$.%%

In the proof of
Theorem \ref{a>1} we consider the function $f_{\varphi}$, which
is defined by the formula 
\begin{equation}\label{f^}
\widehat{f}_{\varphi}(\xi )= | \xi | ^{-n} (\log | \xi | )^{-3/4}
\sum_{j=1}^{\infty} \chi _j(\xi)e^{- i( x_{j} \cdot \xi + t_{j}
{\varphi}(\xi))},
\end{equation}
where $\chi_j$ is the characteristic function of
\begin{equation}\label{omega} \Omega_j =\{ \xi \in \RR;
R_j<|\xi|<R'_j\}.
\end{equation}

\par
\vspace{0.5 cm} Here $(R_j)_1^{\infty}$ and
$(R'_j)_1^{\infty}$ are sequences in $\rr{}$ which fulfill the
following conditions:
\begin{enumerate}
\item $R_1 \geq 2+R$, $R'_1\geq R_1+1$, with $R$ given by Theorem \ref{a>1} or Theorem \ref{a=1};\\[1 ex]
\item $R'_j = R_j^N$ when $j\geq 2$, where $N$ is a large positive number and independent of $j$, which is specified
later on;\\[1 ex]
\item $R_j<R'_j <R_{j+1}$, when $j\geq 1$;\\[1 ex]
\item \begin{equation}\label{vaxande}
|\varphi'(r,\omega)|>1 \qquad \text{when} \qquad r\geq R;
\end{equation} \\[1 ex] 
\item for $j\geq 2$
\begin{equation}\label{Rj1<a<2} R_j >\max_{l<j} \frac{
2^j}{t_l-t_j},
\end{equation}
and
\begin{equation}\label{Rj1<a<2.3} \inf_{R_j\leq r\leq R_j'}(\inf_{|\omega|=1}|\varphi'(r, \omega)|) >\max_{l<j} \frac{2|x_l-x_j|}{t_l-t_j};
\end{equation}
\\[1 ex]
\end{enumerate}
\begin{anm} The sequences $(R_j)_1^{\infty}$ and $(R_j')_1^{\infty}$ can be choosen since $\varphi$ satisfies condition \eqref{funktionerna}.
\end{anm}
Furthermore, in order to get convenient approximations of the
operator $S^{\varphi}$, we let
\begin{equation}\label{S^a_mf}
S^{\varphi}_mf(x,t) = \frac{1}{(2 \pi)^n} \int_{|\xi |< R'_m}
e^{i x\cdot \xi}e^{it \varphi(\xi)}\widehat{f}(\xi)\,  d \xi.
\end{equation}
Then
\begin{gather}\label{S^a_mf2}
S^{\varphi}_m f_{\varphi}(x,t) =  \sum_{j=1}^{m}A^{\varphi}_j(x, t),
\end{gather}
where
\begin{gather}\label{A^a_j}
A^{\varphi}_j(x, t) =\frac{1}{(2 \pi)^n} \int_{\Omega_j} e^{i
(x-x_j)\cdot \xi}e^{i(t-t_j) \varphi(\xi)} | \xi | ^{-n} (\log |
\xi | )^{-3/4}\,  d \xi.
\end{gather}
By using polar coordinates we get
\begin{gather}\label{variabelbyte}
A^{\varphi}_j(x_k,t_k) =\frac{1}{(2 \pi)^n} \int_{|\omega| =1}
\Big\{\int_{R_j}^{R'_j} \frac{1}{r(\log r)^{3/4}}
e^{iF_{\varphi}(r,\omega)}\,  d r\Big\}\,  d \sigma (\omega  ),
\end{gather}
where
\begin{equation}\label{Fa}
F_{\varphi}(r,\omega)= r(x_k - x_j) \cdot \omega + (t_k
-t_j)\varphi(r,\omega),
\end{equation}
and $d \sigma (\omega)$ is the euclidean surface measure on the $n-1$-dimensional unit sphere.
By differentiation we get
\begin{equation}\label{Fa',a>2}
F'_{\varphi}(r,\omega)= (x_k - x_j) \cdot \omega + (t_k
-t_j)\varphi'(r,\omega)
\end{equation}
and
\begin{equation}\label{Fbiss}
F''_{\varphi}(r,\omega)= (t_k -t_j)\varphi''(r,\omega).
\end{equation}
Here recall that $F'_{\varphi}(r\omega)=F'_{\varphi}(r,\omega)$ and $F''_{\varphi}(r,\omega)$
denote the first and second orders of derivatives of
$F_{\varphi}(r,\omega)$ with respect to the $r$-variable.

\par
\vspace{0.5 cm}

By integration by parts in the inner integral of \eqref{variabelbyte} we get
\begin{gather}\label{a>=2, integral for n>1}
\int_{R_j}^{R'_j} \frac{1}{r (\log r )^{3/4}} e^{i F_{\varphi}(r,\omega)}\,
d r =
A_{\varphi}-B_{\varphi},
\end{gather}
where
\begin{gather}\label{Del 1 av partialint}
A_{\varphi}=
\Big[\frac{e^{i F_{\varphi}(r,\omega)}}{r (\log r )^{3/4}i F_{\varphi}'(r,\omega)}\Big
]^{R'_j}_{R_j}
\end{gather}
and
\begin{gather}\label{Del 2 av partialint}
B_{\varphi}=\int_{R_j}^{R'_j} \frac{d}{dr}\Big (\frac{1}{r
(\log r )^{3/4}i F_{\varphi}'(r,\omega)}\Big )e^{i F_{\varphi}(r,\omega)}\,  d r
\end{gather}

%=======================================================================================
%=======================================================================================

\section{Proofs}
In this section we prove Theorems \ref{a>1} and \ref{a=1}. We need some preparing lemmas for the proof. In the following lemma we prove that for fixed $ x\in B_k(0)$ there exists sequences $(x_{n_j})_{1}^{\infty}$ and $(t_{n_j})_{1}^{\infty}$ such that 
\begin{equation*}
x_{n_j} \in \{x_{m_k+1}, \dots ,x_{m_{k+1}}\},\qquad \text{ and } \qquad t_{n_j} \in \{t_{m_k+1}, \dots ,t_{m_{k+1}}\}
\end{equation*}
 and $|x_{n_j}-x|<\gamma(t_{n_j})$.
\begin{lemma}\label{dense}
Let $x\in \RR$ be fixed. Then for each $k \geq |x|$ there exists
$x_{n_j} \in \{x_{m_k+1}, \dots ,x_{m_{k+1}}\}$ and $t_{n_j} \in \{t_{m_k+1}, \dots ,t_{m_{k+1}}\}$ such that
$|x_{n_j}-x|<\gamma(t_{n_j}).$ In particular $(x_{n_j},t_{n_j})\to
(x,0)$ as $j$ turns to infinity.
\end{lemma}
%=======================================================================================
\begin{proof}
For each $k\geq |x|$, $x$ belongs to a cube with vertices in
$T_k=B_{k+1}(0)\cap \delta_{k+1}\mathbf{Z}^n$ and side lengths
$\gamma(1/(k+2))/\sqrt{n}$. Take a vertex $x'$ in the cube and its
diagonal $\gamma(1/(k+2))$ as center and radius of a ball
respectively. This ball $B_{\gamma(1/(k+2))}(x')$ contains the
whole cube and hence also $x$. Therefore there exists $x_{n_j}$
for every $k\geq |x|$ such that $x \in
B_{\gamma(1/(k+2))}(x_{n_j})\subset B_{\gamma(t_{n_j})}(x_{n_j})$.
This proves the first part of the assertion, and the second
statement follows from the fact that $\gamma(0) =0$ and $\gamma$
is continuous and strictly increasing.
\end{proof}

%=======================================================================================

\vspace{0.5 cm}

We want to prove that $f_{\varphi}$ in \eqref{f^} belongs to
$H^{n/2}(\RR)$ and fulfill \eqref{Sa till infty}. The former
relation is a consequence of Lemma \ref{ghat} below, which
concerns Sobolev space properties for functions of the form
\begin{equation}\label{definitiong}
\widehat{g}(\xi )= | \xi | ^{-n} (\log | \xi | )^{-\rho/2}
\sum_{j=1}^{\infty} \chi _j(\xi)b_j(\xi),
\end{equation}
where $\chi_j$ is the characteristic function on disjoint sets $\Omega_j$.
%==========================================================================================================================
%=======================================================================================

\begin{lemma} \label{ghat} Assume that $\rho >1$, $\Omega_j$ for $j \in \mathbf{N}$ are disjoint open subsets of
$\RR\backslash B_{\rho}(0)$ , $b_j\in L^1_{loc}(\RR)$ for $j \in
\mathbf{N}$ satisfies
\begin{equation}
\sup_{j\in \mathbf{N}}\|b_j\|_{L^{\infty}(\Omega_j)}<\infty ,
\end{equation}
and let $\chi_j$ be the characteristic function for $\Omega_j$. If $g$ is given by \eqref{definitiong}, then
$g\in H^{n/2}(\RR)$.
\end{lemma}
%=======================================================================================
\begin{proof} By estimating \eqref{sobolev} for the function $g$ we get that
\begin{multline*}
\int_{\RR}|\widehat{g}(\xi )|^2(1+|\xi|^2)^{n/2}\,  d\xi \cr \leq
C \int_{\RR \setminus B_{\rho}(0)}| \xi | ^{-2n} (\log | \xi | )^{-\rho}
(1+|\xi|^2)^{n/2}\,  d\xi  \cr \leq 2^{n/2}C \int_{\rho}^{\infty}\frac{1}{r(\log r  )^{\rho}} \,  dr < \infty .
\end{multline*}
The second inequality holds since $(1+r^2)^{n/2}<(r^2+r^2)^{n/2}= 2^{n/2}
r^n$ for $r>1$.
\end{proof}
%=======================================================================================
%=======================================================================================
In the following lemma we give estimates of the expression $A_j^{\varphi}$.
%=======================================================================================
%=======================================================================================
\begin{lemma}\label{uppskattning} Let $A_j^{\varphi}(x,t)$ be given by \eqref{A^a_j}.
Then the following is true:
\begin{alignat*}{2}
(1)& \qquad \sum_{j=1}^{k-1} |A^{\varphi}_j(x, t)| \leq C (\log
R'_{k-1})^{1/4}, \textit{ with } C \textit{ independent of } k;
\\[1 ex] (2)& \qquad A^{\varphi}_k(x_k,t_k)>c(\log
R'_k)^{1/4}, \textit{ with } c>0 \textit{ independent of } k.
\end{alignat*}
\end{lemma}
%=======================================================================================
\begin{proof}
$(1)$ By triangle inequality and the fact that $|\xi |> 2$, when
$\xi\in \Omega_j$, we get
\begin{multline*}
 \sum_{j=1}^{k-1} |A^{\varphi}_j(x, t)| \leq
\frac{1}{(2 \pi)^n}\int_{2 \leq |\xi | \leq R'_{k-1}} | \xi |^{-n}
(\log | \xi |)^{-3/4}\,  d \xi\cr = C
\int_2^{R'_{k-1}}\frac{1}{r(\log r)^{3/4}}\,  d r \leq C (\log
R'_{k-1})^{1/4},
\end{multline*}
where $C$ is independent of $k$. In the last equality we have
taken polar coordinates as new variables of integration.
\vspace{0.5 cm}

 %%%%%%%%%%%%%%%%%%%%%%%%%%%%%%%%%%%%%%%%%% %%%%%%%%%%%%%%%%%%%%%%%%%%%%%%%%%%%%%%%%%%

\flushleft{$(2)$} Since $R^N_j=R'_j$ for sufficiently large $N$,
we get
\begin{multline*}
A^{\varphi}_k(x_k,t_k)= C \int_{R_k}^{R'_k} \frac{1}{r(\log
r)^{3/4}}\,  d r \cr = C \Big ((\log R'_k)^{1/4}-(\log
(R'_k)^{1/N})^{1/4}\Big ) \cr = C\Big(1-\frac{1}{N^{1/4}}\Big)
(\log R'_k)^{1/4}
>c(\log R'_k)^{1/4},
\end{multline*}
for some constant $c>0$, which is independent of $k$.
\end{proof}
%=======================================================================================
%=======================================================================================

\begin{lemma}\label{kontinuitet} Assume that $S^{\varphi}_m f_{\varphi}$ is given by \eqref{S^a_mf}. Then $S^{\varphi}_m f_{\varphi}$ is continuous on
$\{(x,t);t>0, x\in \mathbf R^n\}$.
\end{lemma}
%=======================================================================================
\begin{proof}
 The continuity for each $S^{\varphi}_m f_{\varphi}$ follows from the facts, that for almost every $\xi\in \RR$, the map 
 \begin{equation*}
 (x,t)\mapsto e^{i x\cdot \xi}e^{i t {\varphi}(\xi)}
\widehat{f}_{\varphi} ( \xi )
\end{equation*}
is continuous, and that 
 \begin{equation*}\int_{|\xi|<R'_m}
|e^{i x\cdot \xi}e^{i t {\varphi}(\xi)} \widehat{f}_{\varphi} ( \xi )| \, d\xi =
\int_{|\xi|<R'_m} |\widehat{f}_{\varphi} ( \xi )| \, d\xi <C.
\end{equation*}
\end{proof}
%=======================================================================================
%=======================================================================================

When proving Theorem \ref{a>1}, we first prove that the modulus
of $S^{\varphi}_mf_{\varphi}(x_k, t_k)$ turns to infinity as $k$
goes to infinity. For this reason we note that the 
triangle inequality and \eqref{S^a_mf2} implies that
\begin{multline}\label{summa}
|S^{\varphi}_mf_{\varphi}(x_k, t_k)|\geq \cr
|A^{\varphi}_k(x_k,t_k)|-\Big|\sum_{j
=1}^{k-1}A^{\varphi}_j(x_k,t_k)\Big| -\Big|\sum_{j=k+1}^m
A^{\varphi}_j(x_k,t_k) \Big|.
\end{multline}
We want to estimate the terms in \eqref{summa}. From Lemma
\ref{uppskattning} we get estimates for the first two terms. It
remains to estimate the last term.
\vspace{0.5 cm}

\begin{proof}[Proof of Theorem \ref{a>1}.]$\ $

\vspace{-0.15 cm}

{\flushleft{\textbf{Step $\mathbf{1}$.}}}  For $j>k \geq
2$ we shall estimate $|A_j^{\varphi} (x_k,t_k)|$ in \eqref{variabelbyte}. 
We have to find appropriate estimates for $A_{\varphi}$ and
$B_{\varphi}$ in \eqref{a>=2, integral for n>1}-\eqref{Del 2 av partialint}. By using $t_k - t_j
> 0$ and $ \ R_j< r <R'_j$ it follows from \eqref{Rj1<a<2.3}, \eqref{Fbiss}, triangle inequality and Cauchy-Schwarz inequality that
\begin{multline}\label{absfprima}
|F'_{\varphi}(r,\omega)| \geq  (t_k -t_j)
|\varphi'(r,\omega)|-|x_k - x_j|\cr
>(t_k -t_j)|\varphi'(r,\omega)|-(t_k -t_j)\frac{|\varphi'(r,\omega)|}{2}\cr= \frac{|\varphi'(r,\omega)|}{2}(t_k -t_j).
\end{multline}
From \eqref{vaxande}, \eqref{Rj1<a<2} and
\eqref{absfprima} it follows that
\begin{multline*}
|A_{\varphi}| = \Big |  \Big [\frac{1}{r (\log r )^{3/4}i
F'_{\varphi}(r,\omega)}e^{i F_{\varphi}(r,\omega)}\Big
]^{R'_j}_{R_j}\Big | \cr  \leq \frac{C}{R_j}\Big(\frac{1}{|F'_{\varphi}(R_j,\omega)|}+ \frac{1}{|F'_{\varphi}(R'_j,\omega)
|}\Big)\leq \frac{C}{(t_k-t_j)R_j} \leq C 2^{-j}.
\end{multline*}

\par
\vspace{0.5 cm}
In order to estimate $B_{\varphi}$, using \eqref{deriverade funktioner}, \eqref{Fbiss} and \eqref{absfprima}, we have
\begin{multline*}
 \Big | \frac{d}{dr}\Big (\frac{1}{r (\log r )^{3/4}i
F'_{\varphi}(r,\omega)}\Big )e^{i F_{\varphi}(r,\omega)}\Big |\cr
\leq \frac{C}{r^2|F'_{\varphi}(r,\omega)|}  +
\frac{C|F''_{\varphi}(r,\omega)|}{r|F'_{\varphi}(r,\omega)|^2(\log
r )^{3/4}} < \frac{C}{r^2(t_k-t_j)}.
\end{multline*}
This together with \eqref{Rj1<a<2} gives us
\begin{multline*}
|B_{\varphi}| = \Big |\int_{R_j}^{R'_j} \frac{d}{dr}\Big
(\frac{1}{r (\log r )^{3/4}i F'_{\varphi}(r,\omega)}\Big )e^{i
F_{\varphi}(r,\omega)}\,  d r \Big | \cr \leq
\int_{R_j}^{R'_j}\frac{C}{r^2(t_k-t_j)}\,  d r \leq
\frac{C}{R_j(t_k-t_j)} \leq C 2^{-j}.
\end{multline*}
From the estimates above and the triangle inequality we get
\begin{gather}\label{Ajxk}
|A^{\varphi}_j(x_k, t_k)|\leq
C(|A_{\varphi}|+|B_{\varphi}|)< C2^{-j}, \qquad j> k \geq 2.
\end{gather}
 Here $C$ is independent of $j$ and $k$.
\vspace{0.5 cm}
\medspace

 Using the results from
\eqref{summa}, \eqref{Ajxk}, in combination with Lemma
\ref{uppskattning}, and recalling that $R_j'=R_j^N$, gives us
\begin{multline}\label{slutresultat}
|S^{\varphi}_m f_{\varphi}(x_k, t_k)| \geq c (\log
R'_k)^{1/4}- C'(\log R_k)^{1/4}- C\sum _{k+1}^m 2^{-j}\cr \geq c(\log(R_k'))^{1/4}-\frac{C'}{N^{1/4}}(\log(R_k'))^{1/4}-C \geq c
(\log R'_k)^{1/4},
\end{multline}
 when $m>k$ and $N$ is chosen sufficiently large. Here $c>0$ is independent of $k$.

\vspace{0.5 cm}

{\flushleft{\textbf{Step $\mathbf{2}$.} }} Now it remains to show
that $S^{\varphi} f_{\varphi}$ is continuous when $t>0$, and then
it suffices to prove this continuity on a compact subset $L$ of
$$\{(x,t);\, t>0,\, x\in \RR\}.$$ We want to replace $(x_l, t_l)$
with $(x,t)\in L $ in \eqref{Rj1<a<2} and \eqref{Rj1<a<2.3}. Since we have maximum over all $l$ less than $j$, we can choose 
$j_0<\infty$ large enough such that for all $j>l>j_0$ we have that $t_j<t_l<t$. Hence we may replace $(x_l, t_l)$
with $(x,t)\in L $ on the right-hand sides in \eqref{Rj1<a<2} and \eqref{Rj1<a<2.3} for all $j>j_0$. 
This in turn implies that \eqref{Ajxk} holds when $(x_k, t_k)$ is replaced by $(x,
t) \in L$ and $j>j_0$. We use \eqref{Ajxk} to conclude that
\begin{multline*}
|S^{\varphi}_m f_{\varphi}(x,t) -S^{\varphi} f_{\varphi}(x,t)|\cr = \Big
|(2 \pi)^{-n} \int_{|\xi |< R'_m} e^{i x\cdot \xi}e^{it
\varphi(\xi)}\widehat{f}_{\varphi}(\xi)\,  d \xi - (2
\pi)^{-n} \int_{\RR} {e^{i x\cdot \xi}e^{it
\varphi(\xi)}\widehat{f}_{\varphi}(\xi)}\,  d \xi \Big |\cr
=(2 \pi)^{-n} \Big | \int_{|\xi |> R'_m} e^{i x\cdot
\xi}e^{it \varphi(\xi)}\widehat{f}_{\varphi}(\xi)\,  d \xi \Big |
\leq C\sum_{i=m+1}^{\infty}  2^{-i} = C2^{-m},
\end{multline*}
when $m> j_0$.
Hence $S^{\varphi}_m f_{\varphi}$ converge uniformly to $S^{\varphi} f_{\varphi}$ on every compact set.
\par
\vspace{0.5 cm}

We have now showed that
$S^{\varphi}_m f_{\varphi}$ converge uniformly to $S^{\varphi}
f_{\varphi}$ on every compact set and from Lemma \ref{kontinuitet}
it follows that each $S^{\varphi}_m f_{\varphi}$ is a continuous
function. Therefore it follows that $S^{\varphi} f_{\varphi}$ is
continuous on $\{(x,t);t>0\}$. In particular there is an $N \in
\mathbf{N}$ such that 
\begin{equation*}|S^{\varphi}_m
f_{\varphi}(x_k, t_k)-S^{\varphi}
f_{\varphi}(x_k,t_k)| < 1,
\end{equation*}
when $m>N$. 
Using \eqref{slutresultat}
and the triangle inequality we get 
\begin{multline*} c(\log R'_k)^{1/4} \leq
|S^{\varphi}_m f_{\varphi}(x_k, t_k) | \cr \leq
|S^{\varphi}_m f_{\varphi}(x_k, t_k)-S^{\varphi}
f_{\varphi}(x_k,t_k)| + |S^{\varphi} f_{\varphi}
(x_k,t_k)| <\cr 1+ |S^{\varphi}
f_{\varphi}(x_k,t_k)|.
\end{multline*}
This gives us 
\begin{equation*}
|S^{\varphi} f_{\varphi}(x_k,t_k)| > c(\log R'_k)^{1/4}-1
\rightarrow +\infty \text{ \ as \ } k \rightarrow +\infty.
\end{equation*}
For any fixed $x\in \RR$ we can by Lemma \ref{dense} choose a subsequence $(x_{n_j},t_{n_j})$ of $(x_k,t_k)$ that goes to $(x,0)$ as $j$ turns to infinity. This gives the result.
\end{proof}

\vspace{0.2 cm}
%%%%%%%%%%%%%%%%%%%%%%%%%%%%%%%%%%%%%%%%%% %%%%%%%%%%%%%%%%%%%%%%%%%%%%%%%%%%%%%%%%%%
In order to prove Theorem \ref{a=1} we first let $x\in \RR$ be fixed, and consider a modified sequence of $\gamma(t_j)$.
 More precisely, let
\begin{equation}\label{eta}
\eta(t)=\min (h/4,1)\min(\gamma(t), t),
\end{equation}
where $h$ is given by
\eqref{nedat begr}. Then $\eta$ is continuous, strictly
increasing and $\eta(0)=0$. By Lemma \ref{dense} there exist subsequences $(x_{n_j})_1^{\infty}$ and $(t_{n_j})_1^{\infty}$ of $(x_j)_1^{\infty}$ and $(t_j)_1^{\infty}$ respectively such that
\begin{equation*}
|x_{n_j}-x|<\eta(t_{n_j})<\gamma(t_{n_j}).
\end{equation*}
Since $t_{n_j}$ goes to $0$ as $j$ turns to infinity, it follows from \eqref{eta} that 
\begin{equation*}
t_{p_j}-t_{p_{j+1}}\geq (3/h)\eta(t_{p_j}),
\end{equation*}
for some subsequence $(t_{p_j})_{j=1}^{\infty} $ of $(t_{n_j})_{j=1}^{\infty}$.
In the proof of Theorem \ref{a=1} we modify $f_{\varphi}$ in \eqref{f^} into
\begin{equation}\label{f^1}
\widehat{f}_{\varphi,1}(\xi )\equiv | \xi | ^{-n} (\log | \xi | )^{-3/4}
\sum_{j=1}^{\infty} \chi _j(\xi)e^{- i( x_{p_j} \cdot \xi + t_{p_j}
{\varphi}(\xi))},
\end{equation}
where $\chi_j$ is the characteristic function of
\begin{equation*} \Omega_j =\{ \xi \in \RR;
R_j<|\xi|<R'_j\}.\end{equation*}
For $\varphi$ satisfing \eqref{deriverade funktioner} and \eqref{nedat begr} and $j>2$, we replace \eqref{Rj1<a<2} and \eqref{Rj1<a<2.3} by
\begin{equation}\label{Rja=1,1}
R_j >\max_{l<j} \frac{ 2^j}{t_{p_l}-t_{p_j}}.
\end{equation}
If instead $\varphi$ satisfies \eqref{nedat begr} and \eqref{deriverade funktioner1}, then for $j>2$, we replace \eqref{Rj1<a<2} and \eqref{Rj1<a<2.3} by
\begin{equation}\label{Rja=1}  R_j> \Big(
2^{j+2}\max\Big(\frac{1}{t_{p_j}},\frac{1}{\gamma(t_{p_j})}\Big)\Big)^{1/\beta}.
\end{equation}\\[1 ex]

We give the proof of Theorem \ref{a=1} separately depending on which of the conditions \eqref{deriverade funktioner} and \eqref{deriverade funktioner1}, the function $\varphi$ satisfies.

\begin{proof}[Proof of Theorem \ref{a=1} in the case where $\varphi$ satisfies 
\eqref{deriverade funktioner}.]$\ $
\vspace{-0.15 cm}

{\flushleft{\textbf{Step $\mathbf{1}$.}}}
For $j>k \geq 2$ we shall estimate $|A^{\varphi}_j(x_{p_j},t_{p_j})|$ in \eqref{variabelbyte}, where $F_{\varphi}(r,\omega)$ in \eqref{Fa} is replaced by 
\begin{equation}\label{F1}
F_{\varphi ,1}(r,\omega)= r(x_{p_k} - x_{p_j}) \cdot \omega + (t_{p_k}
-t_{p_j})\varphi(r,\omega).
\end{equation} 
We have to find appropriate estimates for $A_{\varphi}$ and $B_{\varphi}$ in \eqref{Del 1 av partialint} and 
\eqref{Del 2 av partialint}, with $F_{\varphi ,1}(r,\omega)$ instead of $F_{\varphi}(r,\omega)$. Since
\begin{equation*}
t_1>t_2>\cdots >0
\end{equation*} and 
\begin{equation*}
t_{p_j}-t_{p_{j+1}}
\geq (3/h)\eta(t_{p_j}),
\end{equation*}
 we have that 
\begin{equation*}
t_{p_k}-t_{p_j}\geq
(3/h)\eta(t_{p_k}).
\end{equation*}
 Using this together with 
\begin{equation*}
|x_{p_k} -
x_{p_j}|\leq |x_{p_k} -x|+|x - x_{p_j}|\leq 2 \eta(t_{p_k}),
\end{equation*} it
follows by the triangle inequality and Cauchy-Schwarz
inequality that
\begin{multline}\label{absfprima1}
|F'_{\varphi,1}(r,\omega)| \geq  (t_{p_k} -t_{p_j})
|\varphi'(r,\omega)|-|x_{p_k} - x_{p_j}|\cr
>(t_{p_k} -t_{p_j})|\varphi'(r,\omega)|-2\eta(t_{p_k})
\geq (t_{p_k}
-t_{p_j})\Big(|\varphi'(r,\omega)|-\frac{2h}{3}\Big)\cr\geq
(t_{p_k} -t_{p_j})\frac{|\varphi'(r,\omega)|}{3}.
\end{multline}
From  \eqref{deriverade funktioner}, \eqref{Rja=1,1} and \eqref{absfprima1}
it follows that
\begin{multline*}
|A_{\varphi}| = \Big |  \Big [\frac{1}{r (\log r )^{3/4}i
F'_{\varphi,1}(r,\omega)}e^{i F_{\varphi ,1}(r,\omega)}\Big
]^{R'_j}_{R_j}\Big | \cr  \leq \frac{C}{R_j}\Big(\frac{1}{|F'_{\varphi,1}(R_j,\omega)|}+\frac{1}{|F'_{\varphi,1}(R'_j,\omega)|}\Big)\leq \frac{C}{(t_{p_k}-t_{p_j})hR_j} \leq C 2^{-j}.
\end{multline*}
\par
\vspace{0.5 cm}

In order to estimate $B_{\varphi}$, we have 
\begin{multline*}
 \Big | \frac{d}{dr}\Big (\frac{1}{r (\log r )^{3/4}i
F'_{\varphi,1}(r,\omega)}\Big )e^{i F_{\varphi ,1}(r,\omega)}\Big |\cr
\leq \frac{C}{r^2|F'_{\varphi,1}(r,\omega)|}  +
\frac{C|F''_{\varphi,1}(r,\omega)|}{r|F'_{\varphi,1}(r,\omega)|^2(\log
r)^{3/4}} \cr
\leq \frac{C}{r^2(t_{p_k}-t_{p_j})|\varphi'(r,\omega)|}  +
\frac{C|\varphi''(r,\omega)|}{r(t_{p_k}-t_{p_j})|\varphi'(r,\omega)|^2(\log
r)^{3/4}}\cr
< \frac{C}{r^{2}(t_{p_k}-t_{p_j})}.
\end{multline*}
This together with \eqref{Rja=1,1} gives us
\begin{multline*}
|B_{\varphi}| = \Big |\int_{R_j}^{R'_j} \frac{d}{dr}\Big
(\frac{1}{r (\log r )^{3/4}i F'_{\varphi,1}(r,\omega)}\Big )e^{i
F_{\varphi ,1}(r,\omega)}\,  d r \Big | \cr \leq
\int_{R_j}^{R'_j}\frac{C}{r^{2}(t_{p_k}-t_{p_j})}\,  d r \leq
\frac{C}{R_j(t_{p_k}-t_{p_j})} \leq C 2^{-j}.
\end{multline*}
From the estimates above and the triangle inequality we get
\begin{gather}\label{Ajxk1der}
|A^{\varphi}_j(x_{p_k}, t_{p_k})|\leq C(|A_{\varphi}|+|B_{\varphi}|)< C2^{-j},\qquad j> k \geq 2.
\end{gather}
 Here $C$ is independent of $j$ and $k$. 
\vspace{0.5 cm}
\medspace

Using the results from \eqref{summa},
\eqref{Ajxk1der}, in combination with
Lemma \ref{uppskattning}, and recalling that $R_j'=R_j^N$, gives
\begin{multline}\label{slutresultat1der}
|S^{\varphi}_m f_{\varphi ,1}(x_{p_k}, t_{p_k})| \geq c (\log
R'_k)^{1/4}- C'(\log R_k)^{1/4}- C\sum _{k+1}^m 2^{-j}\cr  \geq c (\log R'_k)^{1/4},
\end{multline}
 when $m>k$ and $N$ is sufficiently large. Here $c>0$ is independent of $k$.
\vspace{0.5 cm}

{\flushleft{\textbf{Step $\mathbf{2}$.}}} By similar arguments as in the last part of the proof of Theorem \ref{a>1} it follows that $S^{\varphi}f_{\varphi ,1}$ is continuous on each compact subset of 
\begin{equation*}
L=\{(x,t); \, t>0\}
\end{equation*}
and 
\begin{equation*}
|S^{\varphi}f_{\varphi ,1}(x_{p_k},t_{p_k})|\rightarrow +\infty \quad \text{ as } \quad k \rightarrow +\infty.
\end{equation*}
This gives the result for $\varphi$ satisfying
\eqref{deriverade funktioner}.
\end{proof}

\vspace{0.2 cm}
%%%%%%%%%%%%%%%%%%%%%%%%%%%%%%%%%%%%%%%%%% %%%%%%%%%%%%%%%%%%%%%%%%%%%%%%%%%%%%%%%%%%

\begin{proof}[Proof of Theorem \ref{a=1} in the case where $\varphi$ satisfies 
\eqref{deriverade funktioner1}.]$\ $
%%%%%%%%%%%%%%%%%%%%%   %%%%%%%%%%%%%%%%%%%%%%%%%%%%%%%%%%%%%%%%%%
\vspace{-0.15 cm} {\flushleft{\textbf{Step $\mathbf{1}$.}}} For $j>k\geq 2$ we estimate $A^{\varphi}_j(x_{p_k},t_{p_k})$. Let $\eta$, $f_{\varphi ,1}$ and 
$F_{\varphi ,1}(r,\omega)$ be defined by \eqref{eta}, \eqref{f^1} and \eqref{F1} respectively. Since 
\begin{equation*}
t_1>t_2>\cdots >0
\end{equation*} 
and 
\begin{equation*}
t_{p_j}-t_{p_{j+1}}
\geq (3/h)\eta(t_{p_j}),
\end{equation*} we have that 
\begin{equation*}t_{p_k}-t_{p_j}\geq
(3/h)\eta(t_{p_k}).
\end{equation*} 
Using this together with 
\begin{equation*} |x_{p_k} -
x_{p_j}|\leq |x_{p_k} -x|+|x - x_{p_j}|\leq 2 \eta(t_{p_k}),
\end{equation*}  it
follows by the triangle inequality and Cauchy-Schwarz
inequality that
\begin{multline}\label{F'_1}
|F'_{\varphi,1}(r,\omega)|= |(x_{p_k} - x_{p_j}) \cdot \omega +
\varphi'(r,\omega)(t_{p_k} -t_{p_j}) |\cr\geq h(t_{p_k} -t_{p_j})
-|x_{p_k} - x_{p_j}|\geq\eta(t_{p_k})\geq \eta(t_{p_j}).
\end{multline}
Then we estimate each part of equation \eqref{a>=2, integral for n>1} by using \eqref{deriverade funktioner1},  \eqref{Rja=1} and \eqref{F'_1}  and see that
\begin{multline*}
|A_{\varphi}|=\Big |  \Big [\frac{1}{r (\log r )^{3/4}i F'_{\varphi,1}(r,\omega)}e^{i
F_{\varphi ,1}(r,\omega)}\Big ]^{R'_j}_{R_j}\Big |\cr \leq \frac{C}{R_j(\inf_{|\omega|=1}
(|F'_{\varphi,1}(R_j,\omega)|,|F'_{\varphi,1}(R'_j,\omega)|))}<
\frac{C}{R_j\eta(t_{p_j})} < C 2^{-j}.
\end{multline*}
\par
\vspace{0.5 cm}

In order to estimate $B_{\varphi}$, we have 
\begin{multline*}
 \Big | \frac{d}{dr}\Big (\frac{1}{r (\log r )^{3/4}i
F_{\varphi ,1}'(r,\omega)}\Big )e^{i F_{\varphi ,1}(r,\omega)}\Big |\cr
\leq \frac{C}{r^2|F'_{\varphi,1}(r,\omega)|}  +
\frac{C|F''_{\varphi,1}(r,\omega)|}{r|F'_{\varphi,1}(r,\omega)|^2(\log
r)^{3/4} } \cr \leq \frac{C}{r^2\eta(t_{p_j})}  +
\frac{C|\varphi''(r,\omega)|}{r(\log r)^{3/4}\eta(t_{p_j})} \leq
\frac{C}{r^{1+\beta}\eta(t_{p_j})}.
\end{multline*}

This together with \eqref{Rja=1} gives us
\begin{multline*}
|B_{\varphi}|=\Big |\int_{R_j}^{R'_j} \frac{d}{dr}\Big (\frac{1}{r
(\log r )^{3/4}i F_{\varphi ,1}'(r,\omega)}\Big )e^{i
F_{\varphi ,1}(r,\omega)}\, d r \Big | \cr \leq \int_{R_j}^{R'_j}
\frac{C}{r^{\beta+1}\eta(t_{p_j})}\, d r \leq \frac{C}{
R_j^{\beta}\eta(t_{p_j})}<  C2^{-j}.
\end{multline*}
From the estimates above and the triangle inequality we get
\begin{gather}\label{Ajxk1}
|A^{\varphi}_j(x_{p_k}, t_{p_k})|= C(|A_{\varphi}|+|B_{\varphi}|) < C 2^{-j}
\end{gather}
for $j> k \geq 2$. Here $C$ is independent of $j$ and $k$. 
\medspace
\vspace{0.5 cm}

Using the result
from  \eqref{Ajxk1} in combination with
 Lemma \ref{uppskattning}, and recalling that $R_j' = R_j^N$, now gives
\begin{multline}\label{um>clogr1}
|S^{\varphi}_mf_{\varphi ,1}(x_{p_k}, t_{p_k})| \geq c (\log
R'_k)^{1/4}- C'(\log R_k)^{1/4}- C\sum _{k+1}^m 2^{-j}\cr \geq c (\log R'_k)^{1/4}
\end{multline}
when $m>k$ and $N$ is sufficiently large. Here $c>0$ is independent of $k$.
\par
\vspace{0.5 cm}

{\flushleft{\textbf{Step $\mathbf{2}$.}}} By similar arguments as in the last part of the proof of Theorem \ref{a>1} it follows that $S^{\varphi}f_{\varphi ,1}$ is continuous on each compact subset of 
\begin{equation*}
L=\{(x,t); \, t>0\}
\end{equation*}
and 
\begin{equation*}
|S^{\varphi}f_{\varphi ,1}(x_{p_k},t_{p_k})|\rightarrow +\infty \quad \text{ as } \quad k \rightarrow +\infty.
\end{equation*} This gives the result for $\varphi$ satisfying
\eqref{deriverade funktioner1}.
\end{proof}

\end{document}